\documentclass{amsart}
\usepackage{amssymb}
\usepackage{amsmath}
\usepackage{amsfonts}

\newtheorem{theorem}{Theorem}
\theoremstyle{plain}
\newtheorem{acknowledgement}{Acknowledgement}

\newtheorem{corollary}{Corollary}

\newtheorem{definition}{Definition}
\newtheorem{example}{Example}

\newtheorem{lemma}{Lemma}

\numberwithin{equation}{section}

\begin{document}
\title[Approximate fixed point property in IFNS]{Approximate fixed point
property in intuitionistic fuzzy normed space}
\author{M\"{u}zeyyen Ert\"{u}rk}
\address{Department of Mathematics, Adiyaman University, Adiyaman, Turkey}
\email{merturk3263@gmail.com;merturk@yildiz.edu.tr }
\author{Vatan Karakaya}
\address{Ahi Evran \"{U}niversitesi, Ba\u{g}ba\c{s}{\i} Mahallesi, Sahir Kurutluo\u{g}lu Cad. No:100, 40100 Merkez/Kır\c{s}ehir, Turkey}
\email{vkkaya@yahoo.com;vkkaya@yildiz.edu.tr}
\author{M. Mursaleen}
\address{Department of Mathematics, Aligarh Muslim University, Aligarh
202002, India}
\email{mursaleenm@gmail.com}

\begin{abstract}
In this paper, we define concept of approximate fixed point property of a
function and a set in intuitionistic fuzzy normed space. Furthermore, we
give intuitionistic fuzzy version of some class of maps used in fixed point
theory and investigate approximate fixed point\ property of these maps.
\end{abstract}

\maketitle

\section{Introduction}

Fuzzy theory was introduced by Zadeh \cite{Zadeh} and was generalized by
Atanassov \cite{Atanassov} as intuititonistic fuzzy theory. This theory is
used many branches of science. Using idea of intuitionistic fuzzy, Park \cite%
{Park} defined intuitionistic fuzzy metric space, later Saadati and Park
\cite{Saadati} introduced intuitionistic fuzzy normed space. Intuitionistic
fuzzy analogous of \ many concepts used in functional analysis were studied
via intuitionistic fuzzy metric and norm (\cite{karakaya}, \cite{mursaleen},%
\cite{karakaya1}, \cite{mohuiddin}, \cite{samanta},\cite{park1}, \cite{lohan}%
). Fixed point theory is one of fields studied intuitionistic fuzzy version.
Some of works related intuitionistic fuzzy fixed point theory can be found
in \cite{alaca}, \cite{noorani}, \cite{samanta2}, \cite{jes}, \cite{muz}.

On the other hand, there are many problems which can be solved with fixed
point theory. But in most cases, it is enough finding an approximate
solution. So, the existence of fixed point may not be necessary for solution
of\ a problem. A reason of being attractive of this approach is addition of
strong conditions for the solution of problem. To find approximate solution
of problem may be easier putting less requirement. Hence, it is natural to
define approximate fixed point of a function and to produce theory related
to this concept. It is meant that $x$ is close to $f\left( x\right) $ with
approximate fixed point of $f\left( x\right) .$ There are several studies
related to this concept( \cite{dey}, \cite{pacurar}, \cite{anoop}, \cite%
{berinde}, \cite{eva}, \cite{rech}).

In this study, we define and study the concept of approximate fixed point
property of a function and a set which is used fixed point theory in
intuitionistic fuzzy normed space by inspiring studies of Berinde \cite%
{berinde} and Anoop \cite{anoop}. We give examples related this concept in
intuitionistic fuzzy normed space.

Firstly, we mention some concept used in our article.

\begin{definition}[see \protect\cite{schw}]
A binary operation $\ast :$ $\left[ 0,1\right] \times \left[ 0,1\right] $ is
a continuous $t$-norm if it satisfies the following \ conditions: (i) $\ast $
is associative and commutative; (ii) $\ast $ is continuous; (iii) $a\ast 1=a$
for all a $\in $ $\left[ 0,1\right] $; (iv) $a\ast b\leq c\ast d$ whenever $%
a\leq c$ and $b\leq d$ for each $a,b,c,d\in \left[ 0,1\right] $.
\end{definition}

\begin{definition}[see \protect\cite{schw}]
A binary operation $\diamond :\left[ 0,1\right] \times \left[ 0,1\right] $
is a continuous $t$-conorm if it satisfies the following \ conditions: (i) $%
\diamond $ is associative and commutative; (ii) $\diamond $ is continuous;
(iii) $a\diamond 0=a$ for all $a\in $ $\left[ 0,1\right] $; (iv) $a\diamond
b\leq c\diamond d$ whenever $a\leq c$ and $b\leq d$ for each $a,b,c,d\in %
\left[ 0,1\right] $.
\end{definition}

\begin{definition}[see \protect\cite{Saadati}]
Let $\ast $\ be a continuous $t$-norm, $\diamond $ be a continuous $t$%
-conorm and $X$\ be a linear space over the field IF($%
\mathbb{R}
$ or $%
\mathbb{C}
$). If $\mu $ and $\nu $\ are fuzzy sets on $X\times \left( 0,\infty \right)
$ satisfying the following conditions, the five-tuple $\left( X,\mu ,\nu
,\ast ,\diamond \right) $ is said to be an intuitionistic fuzzy normed space
and $\left( \mu ,\nu \right) $ is called an intuitionistic fuzzy norm. For
every $x,y\in X$\ and $s,t>0$,

(i) $\mu \left( x,t\right) +\nu \left( x,t\right) \leq 1$,

(ii) $\mu \left( x,t\right) >0$,

(iii) $\mu \left( x,t\right) =1\Longleftrightarrow x=0$,

(iv) $\mu \left( ax,t\right) =$ $\mu \left( x,\frac{t}{\left\vert
a\right\vert }\right) $\ for each $a\neq 0$,

(v) $\mu \left( x,t\right) \ast $ $\mu \left( y,s\right) \leq $ $\mu \left(
x+y,t+s\right) $

(vi) $\mu \left( x,.\right) :\left( 0,\infty \right) \rightarrow \left[ 0,1%
\right] $\ is continuous,

(vii) $\underset{t\rightarrow \infty }{\lim }\mu \left( x,t\right) =1$ and $%
\underset{t\rightarrow 0}{\lim }\mu \left( x,t\right) =0,$

(viii) $\nu \left( x,t\right) <1$,$\nu $

(ix) $\nu \left( x,t\right) =0\Longleftrightarrow x=0$,

(x) $\nu \left( ax,t\right) =$ $\nu \left( x,\frac{t}{\left\vert
a\right\vert }\right) $\ for each $a\neq 0$,

(xi) $\nu \left( x,t\right) \diamond $ $\nu \left( y,s\right) \geq $ $\nu
\left( x+y,t+s\right) $,

(xii)$\nu \left( x,.\right) :\left( 0,\infty \right) \rightarrow \left[ 0,1%
\right] $ is continuous,

(xiii) $\underset{t\rightarrow \infty }{\lim }\nu \left( x,t\right) =0$ and $%
\underset{t\rightarrow 0}{\lim }\nu \left( x,t\right) =1,$
\end{definition}

we further assume that $\left( X,\mu ,\nu ,\ast ,\diamond \right) $
satisfies the following axiom:

\textit{(xiv)} $\ \ \ \ \ \left.
\begin{array}{c}
a\diamond a=a \\
a\ast a=a%
\end{array}%
\right\} $ for all $a\in \left[ 0,1\right] .$

We use IFNS instead of intuitionistic fuzzy normed space for the sake of
abbreviation.

\begin{lemma}[see \protect\cite{Saadati}]
Let $\left( \mu ,\nu \right) $ be intuitionistic fuzzy norm. For any $t>0,$%
the following hold:
\end{lemma}

\qquad \qquad \textit{(i) }$\mu \left( x,t\right) $\textit{\ and }$\nu
\left( x,t\right) $\textit{\ are nondecreasing and nonincreasing with
respect to }$t$\textit{, respectively.}

\qquad \qquad \textit{(ii) }$\mu \left( x-y,t\right) =\mu \left(
y-x,t\right) $\textit{\ and }$\nu \left( x-y,t\right) =\nu \left(
y-x,t\right) $\textit{.}

\begin{definition}[see \protect\cite{Saadati}]
A sequence $\left( x_{k}\right) $ in $\left( X,\mu ,\nu ,\ast ,\diamond
\right) $ converges to $x$ if and only if
\begin{equation*}
\mu \left( x_{k}-x,t\right) \rightarrow 1\text{ and }\nu \left(
x_{k}-x,t\right) \rightarrow 0
\end{equation*}%
as $k\rightarrow \infty ,$ for each $t>0$. We denote the convergence of $%
\left( x_{k}\right) $ to $x$ by $x_{k}\overset{\left( \mu ,\nu \right) }{%
\rightarrow }x.$
\end{definition}

\begin{definition}[see \protect\cite{Saadati}]
Let $\left( X,\mu ,\nu ,\ast ,\diamond \right) $ be an IFNS. $\left( X,\mu
,\nu ,\ast ,\diamond \right) $ is said to be complete if every Cauchy
sequence in $\left( X,\mu ,\nu ,\ast ,\diamond \right) $ is convergent.
\end{definition}

\begin{definition}[see \protect\cite{mursaleen}]
Let $X$ and $Y$ be two IFNS. $f:X\rightarrow Y$ is continuous at $x_{0}\in X$
if\ $\left( f\left( x_{k}\right) \right) $ in $Y$ convergences to $f\left(
x_{0}\right) $ for any $\left( x_{k}\right) $ in $X$ converging to $x_{0}.$
If $f:X\rightarrow Y$ is continuous at each element of $X,$ then $%
f:X\rightarrow Y$ is said to be continuous on $X.$
\end{definition}

\begin{definition}[see \protect\cite{karakaya1}]
Let $\left( X,\mu ,\nu ,\ast ,\diamond \right) $ be an IFNS. $A\subset X$ is
dense in $\left( X,\mu ,\nu ,\ast ,\diamond \right) $ if \ there exists a
sequence $\left( x_{k}\right) \in A$ such that $x_{k}\overset{\left( \mu
,\nu \right) }{\rightarrow }x$ \ for all $x\in X.$
\end{definition}

\begin{definition}[see \protect\cite{alaca}]
$\left( X,\mu ,\nu ,\ast ,\Diamond \right) $ be an intuitionistic fuzzy
metric space. We call the mapping $f:X\rightarrow X$ intuitionistic fuzzy
contraction map, if there exists $a\in \left( 0,1\right) $ such that%
\begin{equation*}
\mu \left( f\left( x\right) ,f\left( y\right) ,at\right) \geq \mu \left(
x,y,t\right) \text{ and }\nu \left( f\left( x\right) ,f\left( y\right)
,at\right) \leq \mu \left( x,y,t\right)
\end{equation*}%
for all $x,y\in X$ and $t>0.$
\end{definition}

\begin{definition}[see \protect\cite{samanta2}]
$\left( X,\mu ,\nu ,\ast ,\Diamond \right) $ be an intuitionistic fuzzy
metric space. We call the mapping $f:X\rightarrow X$ intuitionistic fuzzy
nonexpansive, if%
\begin{equation*}
\mu \left( f\left( x\right) ,f\left( y\right) ,t\right) \geq \mu \left(
x,y,t\right) \text{ and }\nu \left( f\left( x\right) ,f\left( y\right)
,t\right) \leq \mu \left( x,y,t\right)
\end{equation*}%
for all $x,y\in X$ and $t>0.$
\end{definition}

\begin{lemma}[see \protect\cite{alaca}]
Let $\left( X,\mu ,\nu ,\ast ,\Diamond \right) $ be an intuitionistic fuzzy
metric space.
\end{lemma}

\textit{a) If }$\underset{k\rightarrow \infty }{\lim }x_{k}=x$\textit{\ and }%
$\underset{k\rightarrow \infty }{\lim }y_{k}=y$\textit{\ }%
\begin{equation*}
\mu \left( x,y,t\right) \leq \underset{k\rightarrow \infty }{\lim }\inf \mu
\left( x_{k},y_{k},t\right) \text{ and }\nu \left( x,y,t\right) \geq
\underset{k\rightarrow \infty }{\lim }\sup \nu \left( x_{k},y_{k},t\right)
\end{equation*}%
\textit{for all }$t>0.$

\textit{b) If }$\underset{k\rightarrow \infty }{\lim }x_{k}=x$\textit{\ and }%
$\underset{k\rightarrow \infty }{\lim }y_{k}=y$\textit{\ }%
\begin{equation*}
\mu \left( x,y,t\right) \geq \underset{k\rightarrow \infty }{\lim }\sup \mu
\left( x_{k},y_{k},t\right) \text{ and }\nu \left( x,y,t\right) \leq
\underset{k\rightarrow \infty }{\lim }\inf \nu \left( x_{k},y_{k},t\right)
\end{equation*}%
\textit{for all }$t>0.$

\section{Main Results}

Firstly, we define approximate fixed point property, diameter of a set in
intitionistic fuzzy normed space and give examples.

\begin{definition}
Let $\left( X,\mu ,\nu ,\ast ,\Diamond \right) $ be an IFNS and $%
f:X\rightarrow X$ be a function$.$ Given $\epsilon >0.$ It is said that $%
x_{0}\in X$ is an intuitionistic fuzzy $\epsilon -$fixed point or
approximate fixed point (ifafp) of $f$ if
\begin{equation*}
\mu \left( f(x_{0})-x_{0},t\right) >1-\epsilon \text{ and }\nu \left(
f(x_{0})-x_{0},t\right) <\epsilon
\end{equation*}%
for all $t>0.$ We denote the set of intuitionistic fuzzy $\epsilon -$fixed
points of $f$\ with $F_{\epsilon }^{\left( \mu ,\nu \right) }(f).$
\end{definition}

\begin{definition}
It is said that $f$ has the intuitionistic fuzzy approximate fixed point
property (ifafpp) if $F_{\epsilon }^{\left( \mu ,\nu \right) }(f)$ is not
empty for every $\epsilon >0.$
\end{definition}

\begin{example}
Let $f:%
\mathbb{R}
\rightarrow
\mathbb{R}
$ be a function given by $f\left( x\right) =x+\frac{1}{2}.$ For all $x\in
\mathbb{R}
$ and every $t>0$, $\left(
\mathbb{R}
,\mu ,\nu ,\ast ,\Diamond \right) $ is an intuitionistic fuzzy normed space
with
\begin{equation*}
\mu \left( x,t\right) =\frac{t}{t+\left\vert x\right\vert }\text{ and }\nu
\left( x,t\right) =\frac{\left\vert x\right\vert }{t+\left\vert x\right\vert
}
\end{equation*}%
where $\left\vert .\right\vert $ is usual norm on $%
\mathbb{R}
,$ $a\ast b=a.b$ and $a\Diamond b=\min \left\{ a+b,1\right\} $ for all $%
a,b\in \left[ a,b\right] .$ This function has not any fixed point. For $%
\epsilon >\frac{1}{2t+1}$ and $t>0,$ every $x\in
\mathbb{R}
$ is $\epsilon -$fixed point since
\begin{equation*}
\mu \left( f(x)-x,t\right) =\frac{t}{t+\left\vert f(x)-x\right\vert }%
>1-\epsilon \text{ and }\nu \left( f(x)-x,t\right) =\frac{\left\vert
f(x)-x\right\vert }{t+\left\vert f(x)-x\right\vert },
\end{equation*}%
that is%
\begin{equation*}
\frac{1}{2}<\frac{\epsilon t}{1-\epsilon }.
\end{equation*}%
But $F_{\epsilon }^{\left( \mu ,\nu \right) }(f)=\varnothing $ for $\epsilon
<\frac{1}{2t+1}$ and $t>0.$
\end{example}

\begin{example}
Consider $f(x)=x^{2}$ defined on $\left( 0,1\right) .\left( \left(
0,1\right) ,\mu ,\nu ,\ast ,\Diamond \right) $ is an intuitionistic fuzzy
normed space with respect to $\left( \mu ,\nu \right) $ in the Example 1. As
known, $f$ has not any fixed point on $\left( 0,1\right) .$ We investigate
intuitionistic fuzzy approximate fixed point of $f.$ For every $\epsilon >0$
and$\ t>0,$there exists $x\in \left( 0,1\right) $ such that $x$ satisfies%
\begin{equation*}
\mu \left( f(x)-x,t\right) =\frac{t}{t+\left\vert f(x)-x\right\vert }%
>1-\epsilon \text{ and }\nu \left( f(x)-x,t\right) =\frac{\left\vert
f(x)-x\right\vert }{t+\left\vert f(x)-x\right\vert }
\end{equation*}%
that is%
\begin{equation*}
\left\vert x^{2}-x\right\vert <\frac{\epsilon t}{1-\epsilon }.
\end{equation*}%
So, $f$ has the intuitionistic fuzzy approximate fixed point property, since
$F_{\epsilon }^{\left( \mu ,\nu \right) }(f)$ is not empty for every $%
\epsilon >0.$
\end{example}

\begin{definition}
Let $K$ be nonempty subset of $\left( X,\mu ,\nu ,\ast ,\Diamond \right) .$%
We say that $\left( \delta _{\mu }\left( K\right) ,\delta _{\nu }\left(
K\right) \right) $ is intuitionistic fuzzy diameter of $K$ with respect to $%
t $ where%
\begin{equation*}
\delta _{\mu }\left( K\right) =\inf \left\{ \mu \left( x-y,t\right) :x,y\in
K\right\} \text{ and }\delta _{\nu }\left( K\right) =\sup \left\{ \nu \left(
x-y,t\right) :x,y\in K\right\}
\end{equation*}%
for $t>0.$
\end{definition}

\begin{theorem}
Let $X$ be a intuitionistic fuzzy normed space, and $f:X\rightarrow X$ be a
function. We suppose that
\end{theorem}

\qquad \qquad \qquad \qquad \qquad \qquad \textit{(i)}$F_{\epsilon }^{\left(
\mu ,\nu \right) }(f)\neq \varnothing $

\textit{\qquad \qquad \qquad \qquad \qquad \qquad (ii) There exist }$%
\vartheta \left( \epsilon _{1}\right) ,\vartheta \left( \epsilon _{2}\right)
$\textit{\ such that}%
\begin{eqnarray*}
\mu \left( x-y,t\right) &\geq &\epsilon _{1}\ast \mu \left(
f(x)-f(y),t_{1}\right) \Rightarrow \mu \left( x-y,t\right) \geq \vartheta
\left( \epsilon _{1}\right) ,\forall x,y\in F_{\epsilon }^{\left( \mu ,\nu
\right) }(f) \\
\nu \left( x-y,t\right) &\leq &\epsilon _{2}\Diamond \nu \left(
f(x)-f(y),t_{2}\right) \Rightarrow \nu \left( x-y,t\right) \leq \vartheta
\left( \epsilon _{2}\right) ,\forall x,y\in F_{\epsilon }^{\left( \mu ,\nu
\right) }(f)
\end{eqnarray*}
for $x,y\in F_{\epsilon }^{\left( \mu ,\nu \right) }(f)$ and $\epsilon
_{1},\epsilon _{2}\in \left( 0,1\right) .$

Then $\left( \delta _{\mu }\left( F_{\epsilon }^{\left( \mu ,\nu \right)
}(f)\right) ,\delta _{\nu }\left( F_{\epsilon }^{\left( \mu ,\nu \right)
}(f)\right) \right) =\left( \vartheta \left( \epsilon _{1}\right) ,\vartheta
\left( \epsilon _{2}\right) \right) .$

\begin{proof}
Let $\varepsilon >0$ and $x,y\in F_{\epsilon }^{\left( \mu ,\nu \right)
}(f). $ Then
\begin{equation*}
\mu \left( f(x)-x,t\right) >1-\epsilon \text{ and }\nu \left(
f(x)-x,t\right) <\epsilon
\end{equation*}%
and%
\begin{equation*}
\mu \left( f(y)-y,t\right) >1-\epsilon \text{ and }\nu \left(
f(y)-y,t\right) <\epsilon .
\end{equation*}%
It can be written
\begin{eqnarray*}
\mu \left( x-y,t\right) &\geq &\mu \left( f(x)-x,\frac{t}{3}\right) \ast \mu
\left( f(x)-f(y),\frac{t}{3}\right) \ast \mu \left( f(y)-y,\frac{t}{3}\right)
\\
&\geq &\epsilon \ast \epsilon \ast \mu \left( f(x)-f(y),\frac{t}{3}\right) \\
&=&\epsilon \ast \mu \left( f(x)-f(y),\frac{t}{3}\right)
\end{eqnarray*}%
and%
\begin{eqnarray*}
\nu \left( x-y,t\right) &\leq &\nu \left( f(x)-x,\frac{t}{3}\right) \Diamond
\nu \left( f(x)-f(y),\frac{t}{3}\right) \ast \nu \left( f(y)-y,\frac{t}{3}%
\right) \\
&\leq &\epsilon \Diamond \epsilon \Diamond \nu \left( f(x)-f(y),\frac{t}{3}%
\right) \\
&=&\epsilon \Diamond \nu \left( f(x)-f(y),\frac{t}{3}\right) .
\end{eqnarray*}%
By (ii), for every $x,y\in F_{\epsilon }^{\left( \mu ,\nu \right) }(f),$ we
get
\begin{equation*}
\mu \left( x-y,t\right) \geq \vartheta \left( \epsilon _{1}\right) \text{
and }\nu \left( x-y,t\right) \leq \vartheta \left( \epsilon _{2}\right) .
\end{equation*}%
Hence,
\begin{equation*}
\left( \delta _{\mu }\left( F_{\epsilon }^{\left( \mu ,\nu \right)
}(f)\right) ,\delta _{\nu }\left( F_{\epsilon }^{\left( \mu ,\nu \right)
}(f)\right) \right) =\left( \vartheta \left( \epsilon _{1}\right) ,\vartheta
\left( \epsilon _{2}\right) \right) .
\end{equation*}
\end{proof}

Now, we introduce intuitionistic fuzzy asymptotic regularity to investigate
intuitionistic fuzzy approximate fixed point property of some operators.

\begin{definition}
Let $\left( X,\mu ,\nu ,\ast ,\Diamond \right) $ be an IFNS and $%
f:X\rightarrow X$ be a function. It is said that $f$ is intuitionistic fuzzy
asymptotic regular if
\begin{equation*}
\underset{k\rightarrow \infty }{\lim }\mu \left( f^{k+1}(x)-f^{k}\left(
x\right) ,t\right) =1\text{ and }\underset{k\rightarrow \infty }{\lim }\nu
\left( f^{k+1}(x)-f^{k}\left( x\right) ,t\right) =0
\end{equation*}%
for every $x\in X$ and $t>0.$
\end{definition}

\begin{theorem}
Let $\left( X,\mu ,\nu ,\ast ,\Diamond \right) $ be an IFNS and $%
f:X\rightarrow X$ be a function. If $f$ has intuitionistic fuzzy asymptotic
regular, then there are intuitionistic fuzzy approximate fixed points of $f$.
\end{theorem}

\begin{proof}
Let $x_{0}$ be arbitrary element of $X.$ Since $f$ is intuitionistic fuzzy
asymptotic regular,%
\begin{equation*}
\underset{k\rightarrow \infty }{\lim }\mu \left( f^{k+1}(x_{0})-f^{k}\left(
x_{0}\right) ,t\right) =1\text{ and }\underset{k\rightarrow \infty }{\lim }%
\nu \left( f^{k+1}(x_{0})-f^{k}\left( x_{0}\right) ,t\right) =0.
\end{equation*}%
In this case, for every $\epsilon >0$ there exists $k_{0}\left( \epsilon
\right) \in
\mathbb{N}
$ such that%
\begin{equation*}
\mu \left( f^{k+1}(x_{0})-f^{k}\left( x_{0}\right) ,t\right) >1-\epsilon
\text{ and }\nu \left( f^{k+1}(x_{0})-f^{k}\left( x_{0}\right) ,t\right)
<\epsilon
\end{equation*}%
for every $k\geq k_{0}\left( \epsilon \right) .$ If we denote $f^{k}\left(
x_{0}\right) $ by $y_{0},$ we have%
\begin{eqnarray*}
\mu \left( f^{k+1}(x_{0})-f^{k}\left( x_{0}\right) ,t\right) &=&\mu \left(
f\left( f^{k}(x_{0})\right) -f^{k}\left( x_{0}\right) ,t\right) =\mu \left(
f\left( y_{0})\right) -y_{0},t\right) >1-\epsilon \\
\text{and }\nu \left( f^{k+1}(x_{0})-f^{k}\left( x_{0}\right) ,t\right)
&=&\nu \left( f\left( f^{k}(x_{0})\right) -f^{k}\left( x_{0}\right)
,t\right) =\nu \left( f\left( y_{0})\right) -y_{0},t\right) <\epsilon .
\end{eqnarray*}%
This shows that $y_{0}$ is intuitionistic fuzzy approximate fixed point of $%
f $.
\end{proof}

Now, we introduce intuitionistic fuzzy analogous of mapping such as Kannan,
Chatterjea, Zamfrescu and weak contraction and investigate that these maps
have approximate fixed point under certain conditions. Firstly, we show that
intuitionistic fuzzy contraction map has approximate fixed point property.

\begin{theorem}
Let $\left( X,\mu ,\nu ,\ast ,\Diamond \right) $ be an IFNS, and $%
f:X\rightarrow X$ be a intuiuionistic fuzzy contraction. Then every $%
\epsilon \in \left( 0,1\right) ,F_{\epsilon }^{\left( \mu ,\nu \right)
}(f)\neq \varnothing .$
\end{theorem}

\begin{proof}
Let $x\in X$ and $\epsilon \in \left( 0,1\right) ,$ $t>0.$%
\begin{eqnarray*}
\mu \left( f^{k}(x)-f^{k+1}(x),t\right) &=&\mu \left( f\left(
f^{k-1}(x)\right) -f\left( f^{k}(x)\right) ,t\right) \\
&\geq &\mu \left( f^{k-1}(x)-f^{k}(x),\frac{t}{a}\right) \\
&\geq &\mu \left( f^{k-2}(x)-f^{k-1}(x),\frac{t}{a^{2}}\right) \\
&\geq &... \\
&\geq &\mu \left( x-f(x),\frac{t}{a^{k}}\right)
\end{eqnarray*}%
and%
\begin{eqnarray*}
\nu \left( f^{k}(x)-f^{k+1}(x),t\right) &=&\nu \left( f\left(
f^{k-1}(x)\right) -f\left( f^{k}(x)\right) ,t\right) \\
&\leq &\nu \left( f^{k-1}(x)-f^{k}(x),\frac{t}{a}\right) \\
&\leq &\nu \left( f^{k-2}(x)-f^{k-1}(x),\frac{t}{a^{2}}\right) \\
&\leq &... \\
&\leq &\nu \left( x-f(x),\frac{t}{a^{k}}\right)
\end{eqnarray*}

For $a\in \left( 0,1\right) ,k\rightarrow \infty \Rightarrow \frac{t}{a^{k}}%
\rightarrow \infty ,$ by properties (vii) and (xiii) of intuitionistic fuzzy
norm%
\begin{eqnarray*}
\mu \left( f^{k}(x)-f^{k+1}(x),t\right) &\rightarrow &1 \\
\nu \left( f^{k+1}(x)-f^{k}\left( x\right) ,t\right) &\rightarrow &0.
\end{eqnarray*}%
By Theorem 2, it follows that $F_{\epsilon }^{\left( \mu ,\nu \right)
}(f)\neq \varnothing $ for every $\epsilon \in \left( 0,1\right) $.
\end{proof}

\begin{example}
The open interval $\left( 0,1\right) $ is an intuitionistic fuzzy normed
space with intuitionistic fuzzy norm, $t$-norm and $t$-conorm given in
Example 1. Consider $f:\left( 0,1\right) \rightarrow \left( 0,1\right) $
given by $f\left( x\right) =\frac{1}{2}x.$ This map has not any fixed point
in $\left( 0,1\right) .$ Furthermore, $f$ is an intuitionistic fuzzy
contraction map since%
\begin{eqnarray*}
\mu \left( f(x)-f(y),\frac{t}{2}\right) &=&\frac{\frac{t}{2}}{\frac{t}{2}%
+\left\vert \frac{x}{2}-\frac{y}{2}\right\vert }=\mu \left( x-y,t\right) , \\
\nu \left( f(x)-f(y),\frac{t}{2}\right) &=&\frac{\left\vert \frac{x}{2}-%
\frac{y}{2}\right\vert }{\frac{t}{2}+\left\vert \frac{x}{2}-\frac{y}{2}%
\right\vert }=\nu \left( x-y,t\right)
\end{eqnarray*}%
for every $x,y\in \left( 0,1\right) $ and $t>0.$ We write $\frac{1}{2}x<%
\frac{\epsilon t}{1-\epsilon }$ from
\begin{eqnarray*}
\mu \left( x-f(x),t\right) &=&\mu \left( x-\frac{x}{2},t\right) \\
&=&\mu \left( \frac{x}{2},t\right) =\frac{t}{t+\frac{x}{2}}>1-\epsilon
\end{eqnarray*}%
and%
\begin{eqnarray*}
\nu \left( x-f(x),t\right) &=&\nu \left( x-\frac{x}{2},t\right) \\
&=&\nu \left( \frac{x}{2},t\right) =\frac{\frac{x}{2}}{t+\frac{x}{2}}%
<\epsilon .
\end{eqnarray*}%
For every $\epsilon \in \left( 0,1\right) $ and $t>0$ there exists $x\in
\left( 0,1\right) $ such that $\frac{1}{2}x<\frac{\epsilon t}{1-\epsilon }.$
So, $f$ has intuitionistic fuzzy approximate fixed point property.
\end{example}

\begin{definition}
Let $X$ be an intuitionistic fuzzy normed space. If there exist $a\in \left(
0,\frac{1}{2}\right) $ such that%
\begin{eqnarray*}
\mu \left( f(x)-f(y),at\right) &\geq &\mu \left( x-f(x),t\right) \ast \mu
\left( y-f(y),t\right) \\
\nu \left( f(x)-f(y),at\right) &\leq &\nu \left( x-f(x),t\right) \Diamond
\nu \left( y-f(y),t\right)
\end{eqnarray*}%
for every $x,y\in X$ and $t>0,$ then $f:X\rightarrow X$ is called
intuitionistic fuzzy Kannan operator.
\end{definition}

\begin{theorem}
Let $\left( X,\mu ,\nu ,\ast ,\Diamond \right) $ be an IFNS havig partial
order relation denoted by $\preceq ,$ where $a\ast b=\min \left\{
a,b\right\} $ and $a\Diamond b=\max \left\{ a,b\right\} $, and $%
f:X\rightarrow X$ be an intuitionistic fuzzy Kannan operator satisfying $%
x\preceq f(x)$ for every $x\in X$. Assume that $\preceq \subset XxX$ hold
following conditions:
\end{theorem}

\qquad \qquad \qquad \qquad \textit{(i)} $\preceq $ is subvector space.

\ \ \ \ \ \ \ \ \ \ \ \ \ \ \ \ \ \ \ \ \ \ \ \ \ \ \ \ \ \ \ \ or\textit{\ }

\qquad \qquad \qquad \qquad \textit{(ii) }$X$ is a complete ordered space.

\textit{If } $\mu \left( .,t\right) $\textit{\ is non-decreasing, }$\nu
\left( .,t\right) $\textit{\ is non-increasing for every }$t\in \left(
0,\infty \right) $ and for every $x\succeq \theta $ ($\theta $ is unit
elemant in vector space $X$), \textit{then every }$\epsilon \in \left(
0,1\right) ,F_{\epsilon }^{\left( \mu ,\nu \right) }(f)\neq \varnothing .$

\begin{proof}
Let $x\in X$ and $\epsilon \in \left( 0,1\right) ,$ $t>0.$ We can write from
$x\preceq f(x)$ for every $x\in X,$%
\begin{equation*}
x\preceq f(x)\preceq f^{2}(x)\preceq f^{3}(x)\preceq ...\preceq
f^{k}(x)\preceq \ldots .
\end{equation*}%
Considering assumptions, we have%
\begin{eqnarray*}
\mu \left( f^{k+1}(x)-f^{k}(x),t\right) &=&\mu (f\left( f^{k}(x)\right)
-f\left( f^{k-1}(x)\right) ,t) \\
&\geq &\mu (f^{k}(x)-f^{k+1}(x),\frac{t}{a})\ast \mu (f^{k-1}(x)-f^{k}(x),%
\frac{t}{a}) \\
&\geq &\mu (f^{k}(x)-f^{k-1}(x),\frac{t}{2a})\ast \mu (f^{k-1}(x)-f^{k+1}(x),%
\frac{t}{2a})\ast \mu (f^{k-1}(x)-f^{k}(x),\frac{t}{a}) \\
&\geq &\mu (f^{k}(x)-f^{k-1}(x),\frac{t}{2a})\ast \mu (f^{k-1}(x)-f^{k+1}(x),%
\frac{t}{2a})\ast \mu (f^{k-1}(x)-f^{k}(x),\frac{t}{2a}) \\
&=&\mu (f^{k}(x)-f^{k-1}(x),\frac{t}{2a})\ast \mu (f^{k-1}(x)-f^{k+1}(x),%
\frac{t}{2a}) \\
&=&\min \left\{ \mu (f^{k}(x)-f^{k-1}(x),\frac{t}{2a}),\mu
(f^{k+1}(x)-f^{k-1}(x),\frac{t}{2a})\right\} \\
&=&\mu (f^{k}(x)-f^{k-1}(x),\frac{t}{2a})
\end{eqnarray*}%
\begin{eqnarray*}
&\geq &\mu \left( f^{k-1}(x)-f^{k}(x),\frac{t}{2a^{2}}\right) \ast \mu
\left( f^{k-2}(x)-f^{k-1}(x),\frac{t}{2a^{2}}\right) \\
&\geq &\mu \left( f^{k-1}(x)-f^{k-2}(x),\frac{t}{4a^{2}}\right) \\
&&\ast \mu \left( f^{k-2}(x)-f^{k}(x),\frac{t}{4a^{2}}\right) \ast \mu
\left( f^{k-2}(x)-f^{k-1}(x),\frac{t}{2a^{2}}\right) \\
&\geq &\mu \left( f^{k-1}(x)-f^{k-2}(x),\frac{t}{4a^{2}}\right) \\
&&\ast \mu \left( f^{k-2}(x)-f^{k}(x),\frac{t}{4a^{2}}\right) \ast \mu
\left( f^{k-2}(x)-f^{k-1}(x),\frac{t}{4a^{2}}\right) \\
&=&\mu \left( f^{k-1}(x)-f^{k-2}(x),\frac{t}{4a^{2}}\right) \ast \mu \left(
f^{k-2}(x)-f^{k}(x),\frac{t}{4a^{2}}\right) \\
&=&\min \left\{ \mu \left( f^{k-1}(x)-f^{k-2}(x),\frac{t}{4a^{2}}\right)
,\mu \left( f^{k}(x)-f^{k-2}(x),\frac{t}{4a^{2}}\right) \right\} \\
&=&\mu \left( f^{k-1}(x)-f^{k-2}(x),\frac{t}{4a^{2}}\right) \\
&&\vdots \\
&\geq &\mu \left( f^{k-(k-2)}(x)-f^{k-(k-1)}(x),\frac{t}{2^{k-1}a^{k-1}}%
\right) \\
&=&\mu \left( f^{2}(x)-f(x),\frac{t}{2^{k-1}a^{k-1}}\right) \\
&\geq &\mu \left( f^{2}(x)-f(x),\frac{t}{2^{k-1}a^{k}}\right) \ast \mu
\left( x-f(x),\frac{t}{2^{k-1}a^{k}}\right) \\
&\geq &\mu \left( f^{2}(x)-x,\frac{t}{2^{k}a^{k}}\right) \ast \mu \left(
x-f(x),\frac{t}{2^{k}a^{k}}\right) \ast \mu \left( x-f(x),\frac{t}{%
2^{k-1}a^{k}}\right) \\
&\geq &\mu \left( f^{2}(x)-x,\frac{t}{2^{k}a^{k}}\right) \ast \mu \left(
x-f(x),\frac{t}{2^{k}a^{k}}\right) \\
&=&\min \left\{ \mu \left( x-f^{2}(x),\frac{t}{2^{k}a^{k}}\right) ,\mu
\left( x-f(x),\frac{t}{2^{k}a^{k}}\right) \right\} \\
&=&\mu \left( x-f(x),\frac{t}{2^{k}a^{k}}\right)
\end{eqnarray*}%
and%
\begin{eqnarray*}
\nu \left( f^{k+1}(x)-f^{k}(x),t\right) &=&\nu (f\left( f^{k}(x)\right)
-f\left( f^{k-1}(x)\right) ,t) \\
&\leq &\nu (f^{k}(x)-f^{k+1}(x),\frac{t}{a})\Diamond \nu
(f^{k-1}(x)-f^{k}(x),\frac{t}{a}) \\
&\leq &\nu (f^{k}(x)-f^{k-1}(x),\frac{t}{2a})\Diamond \nu
(f^{k-1}(x)-f^{k+1}(x),\frac{t}{2a})\Diamond \nu (f^{k-1}(x)-f^{k}(x),\frac{t%
}{a}) \\
&\leq &\nu (f^{k}(x)-f^{k-1}(x),\frac{t}{2a})\Diamond \nu
(f^{k-1}(x)-f^{k+1}(x),\frac{t}{2a})\Diamond \nu (f^{k-1}(x)-f^{k}(x),\frac{t%
}{2a}) \\
&=&\nu (f^{k}(x)-f^{k-1}(x),\frac{t}{2a})\Diamond \nu (f^{k-1}(x)-f^{k+1}(x),%
\frac{t}{2a})
\end{eqnarray*}%
\begin{eqnarray*}
&=&\max \left\{ \nu (f^{k}(x)-f^{k-1}(x),\frac{t}{2a}),\nu
(f^{k+1}(x)-f^{k-1}(x),\frac{t}{2a})\right\} \\
&=&\nu (f^{k}(x)-f^{k-1}(x),\frac{t}{2a}) \\
&\leq &\nu \left( f^{k-1}(x)-f^{k}(x),\frac{t}{2a^{2}}\right) \Diamond \nu
\left( f^{k-2}(x)-f^{k-1}(x),\frac{t}{2a^{2}}\right) \\
&\leq &\nu \left( f^{k-1}(x)-f^{k-2}(x),\frac{t}{4a^{2}}\right) \\
&&\Diamond \nu \left( f^{k-2}(x)-f^{k}(x),\frac{t}{4a^{2}}\right) \Diamond
\nu \left( f^{k-2}(x)-f^{k-1}(x),\frac{t}{2a^{2}}\right) \\
&\leq &\nu \left( f^{k-1}(x)-f^{k-2}(x),\frac{t}{4a^{2}}\right) \\
&&\Diamond \nu \left( f^{k-2}(x)-f^{k}(x),\frac{t}{4a^{2}}\right) \Diamond
\nu \left( f^{k-2}(x)-f^{k-1}(x),\frac{t}{4a^{2}}\right) \\
&=&\nu \left( f^{k-1}(x)-f^{k-2}(x),\frac{t}{4a^{2}}\right) \Diamond \nu
\left( f^{k-2}(x)-f^{k}(x),\frac{t}{4a^{2}}\right) \\
&=&\max \left\{ \nu \left( f^{k-1}(x)-f^{k-2}(x),\frac{t}{4a^{2}}\right)
,\nu \left( f^{k}(x)-f^{k-2}(x),\frac{t}{4a^{2}}\right) \right\} \\
&=&\nu \left( f^{k-1}(x)-f^{k-2}(x),\frac{t}{4a^{2}}\right) \\
&&\vdots \\
&\leq &\nu \left( f^{k-(k-2)}(x)-f^{k-(k-1)}(x),\frac{t}{2^{k-1}a^{k-1}}%
\right) \\
&=&\nu \left( f^{2}(x)-f(x),\frac{t}{2^{k-1}a^{k-1}}\right) \\
&\leq &\nu \left( f^{2}(x)-f(x),\frac{t}{2^{k-1}a^{k}}\right) \Diamond \nu
\left( x-f(x),\frac{t}{2^{k-1}a^{k}}\right) \\
&\leq &\nu \left( f^{2}(x)-x,\frac{t}{2^{k}a^{k}}\right) \Diamond \nu \left(
x-f(x),\frac{t}{2^{k}a^{k}}\right) \Diamond \nu \left( x-f(x),\frac{t}{%
2^{k-1}a^{k}}\right) \\
&\leq &\nu \left( f^{2}(x)-x,\frac{t}{2^{k}a^{k}}\right) \Diamond \nu \left(
x-f(x),\frac{t}{2^{k}a^{k}}\right) \\
&=&\max \left\{ \nu \left( x-f^{2}(x),\frac{t}{2^{k}a^{k}}\right) ,\nu
\left( x-f(x),\frac{t}{2^{k}a^{k}}\right) \right\} \\
&=&\nu \left( x-f(x),\frac{t}{2^{k}a^{k}}\right)
\end{eqnarray*}%
Now, if we take limit for $k\rightarrow \infty ,$ $\frac{t}{\left( 2a\right)
^{k}}$ tends to infinity for $a\in \left( 0,\frac{1}{2}\right) .$ Using
prepoerties (vii) and (xiii) of intuitionistic fuzzy norm$,$%
\begin{eqnarray*}
\underset{k\rightarrow \infty }{\lim }\mu \left(
f^{k}(x)-f^{k+1}(x),t\right) &\geq &\underset{k\rightarrow \infty }{\lim }%
\mu \left( x-f(x),\frac{t}{\left( 2a\right) ^{k}}\right) =1 \\
\underset{k\rightarrow \infty }{\lim }\nu \left(
f^{k}(x)-f^{k+1}(x),t\right) &\leq &\underset{k\rightarrow \infty }{\lim }%
\nu \left( x-f(x),\frac{t}{\left( 2a\right) ^{k}}\right) =0.
\end{eqnarray*}%
This means that intuitioonistic fuzzy Kannan operator is intuitionistic
fuzzy asymptotic regular. That is, intuitioonistic fuzzy Kannan operator has
approximate fixed point.
\end{proof}

\begin{corollary}
In the Theorem 4 , if \textit{\ }$x\succeq f(x)$\textit{\ for intuitionistic
fuzzy Kannan operator }$f$\textit{\ and } $\mu \left( .,t\right) $\textit{\
is non-increasing, }$\nu \left( .,t\right) $\textit{\ is non-decreasing for
every }$t\in \left( 0,\infty \right) $ and for every $x\succeq \theta ,f$
has still intuitionistic fuzzy approximate fixed point property.
\end{corollary}

\begin{definition}
Let $X$ be an intuitionistic fuzzy normed space. If there exist $a\in \left(
0,\frac{1}{2}\right) $ such that%
\begin{eqnarray*}
\mu \left( f(x)-f(y),at\right) &\geq &\mu \left( x-f(y),t\right) \ast \mu
\left( y-f(x),t\right) \\
\nu \left( f(x)-f(y),at\right) &\leq &\nu \left( x-f(y),t\right) \Diamond
\nu \left( y-f(x),t\right)
\end{eqnarray*}%
for every $x,y\in X$ and $t>0,$ then $f:X\rightarrow X$ is called
intuitionistic fuzzy Chatterjea operator.
\end{definition}

\begin{theorem}
Let $\left( X,\mu ,\nu ,\ast ,\Diamond \right) $ be an IFNS havig partial
order relation denoted by $\preceq ,$ where $a\ast b=\min \left\{
a,b\right\} $ and $a\Diamond b=\max \left\{ a,b\right\} $, and $%
f:X\rightarrow X$ be an intuitionistic fuzzy Chatterjea operator satisfying $%
x\preceq f(x)$ for every $x\in X$. Assume that $\preceq \subset XxX$ hold
following conditions:
\end{theorem}

\qquad \qquad \qquad \qquad \textit{(i)} $\preceq $ is subvector space.

\ \ \ \ \ \ \ \ \ \ \ \ \ \ \ \ \ \ \ \ \ \ \ \ \ \ \ \ \ \ \ \ or\textit{\ }

\qquad \qquad \qquad \qquad \textit{(ii) }$X$ is a complete ordered space.

\textit{If } $\mu \left( .,t\right) $\textit{\ is non-decreasing, }$\nu
\left( .,t\right) $\textit{\ is non-increasing for every }$t\in \left(
0,\infty \right) $ and for every $x\succeq \theta $ ($\theta $ is unit
elemant in vector space $X$), \textit{then every }$\epsilon \in \left(
0,1\right) ,F_{\epsilon }^{\left( \mu ,\nu \right) }(f)\neq \varnothing .$

\begin{proof}
By taking into consideration assumption of theorem, we get%
\begin{eqnarray*}
\mu \left( f^{k+1}(x)-f^{k}(x),t\right) &\geq &\mu (f^{k}(x)-f^{k}(x),\frac{t%
}{a})\ast \mu (f^{k-1}(x)-f^{k+1}(x),\frac{t}{a}) \\
&=&1\ast \mu (f^{k-1}(x)-f^{k+1}(x),\frac{t}{a})=\mu (f^{k-1}(x)-f^{k+1}(x),%
\frac{t}{a}) \\
&\geq &\mu (f^{k-1}(x)-f^{k}(x),\frac{t}{2a})\ast \mu (f^{k}(x)-f^{k+1}(x),%
\frac{t}{2a}) \\
&=&\min \left\{ \mu (f^{k-1}(x)-f^{k}(x),\frac{t}{2a}),\mu
(f^{k+1}(x)-f^{k}(x),\frac{t}{2a})\right\} \\
&=&\mu (f^{k-1}(x)-f^{k}(x),\frac{t}{2a}) \\
&\geq &\mu \left( f^{k-2}(x)-f^{k}(x),\frac{t}{2a^{2}}\right) \ast \mu
\left( f^{k-1}(x)-f^{k-1}(x),\frac{t}{2a^{2}}\right) \\
&=&\mu \left( f^{k-2}(x)-f^{k}(x),\frac{t}{2a^{2}}\right) \ast 1=\mu \left(
f^{k-2}(x)-f^{k}(x),\frac{t}{2a^{2}}\right)
\end{eqnarray*}%
\begin{eqnarray*}
&\geq &\mu \left( f^{k-2}(x)-f^{k-1}(x),\frac{t}{4a^{2}}\right) \ast \mu
\left( f^{k-1}(x)-f^{k}(x),\frac{t}{4a^{2}}\right) \\
&=&\min \left\{ \mu \left( f^{k-2}(x)-f^{k-1}(x),\frac{t}{4a^{2}}\right)
,\mu \left( f^{k}(x)-f^{k-1}(x),\frac{t}{4a^{2}}\right) \right\} \\
&=&\mu \left( f^{k-2}(x)-f^{k-1}(x),\frac{t}{4a^{2}}\right) \\
&&\vdots \\
&\geq &\mu \left( f^{k-\left( k-2\right) }(x)-f^{k-\left( k-3\right) }(x),%
\frac{t}{2^{k-2}a^{k-2}}\right) \\
&=&\mu \left( f^{2}(x)-f^{3}(x),\frac{t}{2^{k-2}a^{k-2}}\right) \\
&\geq &\mu \left( f\left( x\right) -f^{3}(x),\frac{t}{2^{k-2}a^{k-1}}\right)
\ast \mu \left( f^{2}(x)-f^{2}(x),\frac{t}{2^{k-2}a^{k-1}}\right) \\
&=&\mu \left( f\left( x\right) -f^{3}(x),\frac{t}{2^{k-2}a^{k-1}}\right)
\ast 1=\mu \left( f\left( x\right) -f^{3}(x),\frac{t}{2^{k-2}a^{k-1}}\right)
\\
&\geq &\mu \left( f\left( x\right) -f^{2}(x),\frac{t}{2^{k-1}a^{k-1}}\right)
\ast \mu \left( f^{2}(x)-f^{3}\left( x\right) ,\frac{t}{2^{k-1}a^{k-1}}%
\right) \\
&=&\min \left\{ \mu \left( f\left( x\right) -f^{2}(x),\frac{t}{2^{k-1}a^{k-1}%
}\right) ,\mu \left( f^{3}\left( x\right) -f^{2}(x),\frac{t}{2^{k-1}a^{k-1}}%
\right) \right\} \\
&=&\mu \left( f\left( x\right) -f^{2}(x),\frac{t}{2^{k-1}a^{k-1}}\right) \\
&\geq &\mu \left( x-f^{2}(x),\frac{t}{2^{k-1}a^{k}}\right) \ast \mu \left(
f\left( x\right) -f(x),\frac{t}{2^{k-1}a^{k}}\right) \\
&=&\mu \left( x-f^{2}(x),\frac{t}{2^{k-1}a^{k}}\right) \ast 1=\mu \left(
x-f^{2}(x),\frac{t}{2^{k-1}a^{k}}\right) \\
&\geq &\mu \left( x-f(x),\frac{t}{2^{k}a^{k}}\right) \ast \mu \left(
f(x)-f^{2}(x),\frac{t}{2^{k}a^{k}}\right) \\
&=&\min \left\{ \mu \left( x-f(x),\frac{t}{2^{k}a^{k}}\right) ,\mu \left(
f^{2}(x)-f(x),\frac{t}{2^{k}a^{k}}\right) \right\} \\
&=&\mu \left( x-f(x),\frac{t}{2^{k}a^{k}}\right)
\end{eqnarray*}%
and%
\begin{eqnarray*}
\nu \left( f^{k+1}(x)-f^{k}(x),t\right) &\leq &\nu (f^{k}(x)-f^{k}(x),\frac{t%
}{a})\Diamond \nu (f^{k-1}(x)-f^{k+1}(x),\frac{t}{a}) \\
&=&0\Diamond \nu (f^{k-1}(x)-f^{k+1}(x),\frac{t}{a})=\nu
(f^{k-1}(x)-f^{k+1}(x),\frac{t}{a}) \\
&\leq &\nu (f^{k-1}(x)-f^{k}(x),\frac{t}{2a})\Diamond \nu
(f^{k}(x)-f^{k+1}(x),\frac{t}{2a}) \\
&=&\max \left\{ \nu (f^{k-1}(x)-f^{k}(x),\frac{t}{2a}),\nu
(f^{k+1}(x)-f^{k}(x),\frac{t}{2a})\right\}
\end{eqnarray*}%
\begin{eqnarray*}
&=&\nu (f^{k-1}(x)-f^{k}(x),\frac{t}{2a}) \\
&\leq &\nu \left( f^{k-2}(x)-f^{k}(x),\frac{t}{2a^{2}}\right) \Diamond \nu
\left( f^{k-1}(x)-f^{k-1}(x),\frac{t}{2a^{2}}\right) \\
&=&\nu \left( f^{k-2}(x)-f^{k}(x),\frac{t}{2a^{2}}\right) \Diamond 0=\nu
\left( f^{k-2}(x)-f^{k}(x),\frac{t}{2a^{2}}\right) \\
&\leq &\nu \left( f^{k-2}(x)-f^{k-1}(x),\frac{t}{4a^{2}}\right) \Diamond \nu
\left( f^{k-1}(x)-f^{k}(x),\frac{t}{4a^{2}}\right) \\
&=&\max \left\{ \mu \left( f^{k-2}(x)-f^{k-1}(x),\frac{t}{4a^{2}}\right)
,\nu \left( f^{k}(x)-f^{k-1}(x),\frac{t}{4a^{2}}\right) \right\} \\
&=&\nu \left( f^{k-2}(x)-f^{k-1}(x),\frac{t}{4a^{2}}\right) \\
&&\vdots \\
&\leq &\nu \left( f^{k-\left( k-2\right) }(x)-f^{k-\left( k-3\right) }(x),%
\frac{t}{2^{k-2}a^{k-2}}\right) \\
&=&\nu \left( f^{2}(x)-f^{3}(x),\frac{t}{2^{k-2}a^{k-2}}\right) \\
&\leq &\nu \left( f\left( x\right) -f^{3}(x),\frac{t}{2^{k-2}a^{k-1}}\right)
\Diamond \nu \left( f^{2}(x)-f^{2}(x),\frac{t}{2^{k-2}a^{k-1}}\right) \\
&=&\nu \left( f\left( x\right) -f^{3}(x),\frac{t}{2^{k-2}a^{k-1}}\right)
\Diamond 0=\nu \left( f\left( x\right) -f^{3}(x),\frac{t}{2^{k-2}a^{k-1}}%
\right) \\
&\leq &\nu \left( f\left( x\right) -f^{2}(x),\frac{t}{2^{k-1}a^{k-1}}\right)
\Diamond \nu \left( f^{2}(x)-f^{3}\left( x\right) ,\frac{t}{2^{k-1}a^{k-1}}%
\right) \\
&=&\max \left\{ \nu \left( f\left( x\right) -f^{2}(x),\frac{t}{2^{k-1}a^{k-1}%
}\right) ,\nu \left( f^{3}\left( x\right) -f^{2}(x),\frac{t}{2^{k-1}a^{k-1}}%
\right) \right\} \\
&=&\nu \left( f\left( x\right) -f^{2}(x),\frac{t}{2^{k-1}a^{k-1}}\right) \\
&\leq &\nu \left( x-f^{2}(x),\frac{t}{2^{k-1}a^{k}}\right) \Diamond \nu
\left( f\left( x\right) -f(x),\frac{t}{2^{k-1}a^{k}}\right) \\
&=&\nu \left( x-f^{2}(x),\frac{t}{2^{k-1}a^{k}}\right) \Diamond 0=\nu \left(
x-f^{2}(x),\frac{t}{2^{k-1}a^{k}}\right) \\
&\leq &\nu \left( x-f(x),\frac{t}{2^{k}a^{k}}\right) \Diamond \nu \left(
f(x)-f^{2}(x),\frac{t}{2^{k}a^{k}}\right) \\
&=&\max \left\{ \nu \left( x-f(x),\frac{t}{2^{k}a^{k}}\right) ,\nu \left(
f^{2}(x)-f(x),\frac{t}{2^{k}a^{k}}\right) \right\} \\
&=&\nu \left( x-f(x),\frac{t}{2^{k}a^{k}}\right)
\end{eqnarray*}%
Since $\frac{t}{\left( 2a\right) ^{k}}\rightarrow \infty $ for $k\rightarrow
\infty ,$ by means of (vii) and (xiii) properties of intuitionistic fuzzy
norm, we see intuitionistic fuzzy Chatterjea operator has approximate fixed
point property.
\end{proof}

\begin{corollary}
In the Theorem 5, if \ \textit{\ }$x\succeq f(x)$\textit{\ for
intuitionistic fuzzy Chatterjea operator }$f$\textit{\ and } $\mu \left(
.,t\right) $\textit{\ is non-increasing, }$\nu \left( .,t\right) $\textit{\
is non-decreasing for every }$t\in \left( 0,\infty \right) $ and for every $%
x\succeq \theta ,f$ has still intuitionistic fuzzy approximate fixed point
property.
\end{corollary}

\begin{definition}
Let $X$ be an intuitionistic fuzzy normed space. A mapping $f:X\rightarrow X$
is called intuitionistic fuzzy Zamfirecsu operator if there exists at least $%
a\in \left( 0,1\right) ,k\in \left( 0,\frac{1}{2}\right) ,c\in \left( 0,%
\frac{1}{2}\right) $ such that at least one of the followings is true for
every $x,y\in X$ and $t>0:$
\end{definition}

\qquad \qquad \qquad \textit{i)}%
\begin{eqnarray*}
\mu \left( f(x)-f(y),at\right) &\geq &\mu \left( x-y),t\right) \\
\nu \left( f(x)-f(y),at\right) &\leq &\nu \left( x)-y),t\right) .
\end{eqnarray*}

\textit{\qquad \qquad \qquad ii)}%
\begin{eqnarray*}
\mu \left( f(x)-f(y),kt\right) &\geq &\mu \left( x-f(x),t\right) \ast \mu
\left( y-f(y),t\right) \\
\nu \left( f(x)-f(y),kt\right) &\leq &\nu \left( x-f(x),t\right) \Diamond
\nu \left( y-f(y),t\right) .
\end{eqnarray*}

\textit{\qquad \qquad \qquad iii)}
\begin{eqnarray*}
\mu \left( f(x)-f(y),ct\right) &\geq &\mu \left( x-f(y),t\right) \ast \mu
\left( y-f(x),t\right) \\
\nu \left( f(x)-f(y),ct\right) &\leq &\nu \left( x-f(y),t\right) \Diamond
\nu \left( y-f(x),t\right) .
\end{eqnarray*}

\begin{theorem}
Let $\left( X,\mu ,\nu ,\ast ,\Diamond \right) $ be an IFNS havig partial
order relation denoted by $\preceq ,$ where $a\ast b=\min \left\{
a,b\right\} $ and $a\Diamond b=\max \left\{ a,b\right\} $, and $%
f:X\rightarrow X$ be an intuitionistic fuzzy Zamfirescu operator satisfying $%
x\preceq f(x)$ for every $x\in X$. Assume that $\preceq \subset XxX$ hold
following conditions:
\end{theorem}

\qquad \qquad \qquad \qquad \textit{(i)} $\preceq $ is subvector space.

\ \ \ \ \ \ \ \ \ \ \ \ \ \ \ \ \ \ \ \ \ \ \ \ \ \ \ \ \ \ \ \ or\textit{\ }

\qquad \qquad \qquad \qquad \textit{(ii) }$X$ is a complete ordered space.

\textit{If } $\mu \left( .,t\right) $\textit{\ is non-decreasing, }$\nu
\left( .,t\right) $\textit{\ is non-increasing for every }$t\in \left(
0,\infty \right) $ and for every $x\succeq \theta $ ($\theta $ is unit
elemant in vector space $X$), \textit{then every }$\epsilon \in \left(
0,1\right) ,F_{\epsilon }^{\left( \mu ,\nu \right) }(f)\neq \varnothing .$

\begin{proof}
The proof is clear from Theorem 4 and Theorem 5.
\end{proof}

\begin{definition}
Let $X$ be an IFNS. If there exist $a\in \left( 0,1\right) $ and $L\geq 0$
such that%
\begin{eqnarray*}
\mu \left( f(x)-f(y),t\right) &\geq &\mu \left( x-y,\frac{t}{a}\right) \ast
\mu \left( y-f(x),\frac{t}{L}\right) \\
\nu \left( f(x)-f(y),t\right) &\leq &\nu \left( x-y,\frac{t}{a}\right)
\Diamond \nu \left( y-f(x),\frac{t}{L}\right)
\end{eqnarray*}%
for every $x,y\in X$ and $t>0,$ then $f:X\rightarrow X$ is called
intuitionistic fuzzy weak contraction operator.
\end{definition}

\begin{theorem}
Let $X$ be an IFNS, and $f:X\rightarrow X$ be intuitionistic fuzzy weak$\ $%
contraction. Then every $\epsilon \in \left( 0,1\right) ,$ $F_{\epsilon
}^{\left( \mu ,\nu \right) }(f)\neq \varnothing .$
\end{theorem}

\begin{proof}
Let $x\in X$ and $\epsilon \in \left( 0,1\right) .$%
\begin{eqnarray*}
\mu \left( f^{k}(x)-f^{k+1}(x),t\right) &=&\mu \left( f\left(
f^{k-1}(x)\right) -f\left( f^{k}(x)\right) ,t\right) \\
&\geq &\mu \left( f^{k-1}(x)-f^{k}(x),\frac{t}{a}\right) \ast \mu \left(
f^{k}(x)-f^{k}(x),\frac{t}{L}\right) \\
&=&\mu \left( f^{k-1}(x)-f^{k}(x),\frac{t}{a}\right) \ast 1=\mu \left(
f^{k-1}(x)-f^{k}(x),\frac{t}{a}\right) \\
&\geq &\mu \left( f^{k-2}(x)-f^{k-1}(x),\frac{t}{a^{2}}\right) \ast \mu
\left( f^{k-1}(x)-f^{k-1}(x),\frac{t}{L}\right) \\
&=&\mu \left( f^{k-2}(x)-f^{k-1}(x),\frac{t}{a^{2}}\right) \ast 1 \\
&\geq &\mu \left( f^{k-2}(x)-f^{k-1}(x),\frac{t}{a^{2}}\right) \\
&\geq &... \\
&=&\mu \left( f^{k-(k-1)}(x)-f^{k-(k-2)}(x),\frac{t}{a^{k-1}}\right) =\mu
\left( f(x)-f^{2}(x),\frac{t}{a^{k-1}}\right) \\
&\geq &\mu \left( x-f(x),\frac{t}{a^{k}}\right) \ast \mu \left( f(x)-f(x),%
\frac{t}{L}\right) \\
&\geq &\mu \left( x-f(x),\frac{t}{a^{k}}\right) \ast 1=\mu \left( x-f(x),%
\frac{t}{a^{k}}\right)
\end{eqnarray*}%
and%
\begin{eqnarray*}
\nu \left( f^{k}(x)-f^{k+1}(x),t\right) &=&\nu \left( f\left(
f^{k-1}(x)\right) -f\left( f^{k}(x)\right) ,t\right) \\
&\leq &\nu \left( f^{k-1}(x)-f^{k}(x),\frac{t}{a}\right) \Diamond \nu \left(
f^{k}(x)-f^{k}(x),\frac{t}{L}\right) \\
&=&\nu \left( f^{k-1}(x)-f^{k}(x),\frac{t}{a}\right) \Diamond 0=\nu \left(
f^{k-1}(x)-f^{k}(x),\frac{t}{a}\right) \\
&\leq &\nu \left( f^{k-2}(x)-f^{k-1}(x),\frac{t}{a^{2}}\right) \Diamond \nu
\left( f^{k-1}(x)-f^{k-1}(x),\frac{t}{L}\right) \\
&=&\nu \left( f^{k-2}(x)-f^{k-1}(x),\frac{t}{a^{2}}\right) \Diamond 0 \\
&\leq &\nu \left( f^{k-2}(x)-f^{k-1}(x),\frac{t}{a^{2}}\right) \\
&\leq &... \\
&=&\nu \left( f^{k-(k-1)}(x)-f^{k-(k-2)}(x),\frac{t}{a^{k-1}}\right) =\nu
\left( f(x)-f^{2}(x),\frac{t}{a^{k-1}}\right) \\
&\leq &\nu \left( x-f(x),\frac{t}{a^{k}}\right) \Diamond \nu \left(
f(x)-f(x),\frac{t}{L}\right) \\
&\leq &\nu \left( x-f(x),\frac{t}{a^{k}}\right) \Diamond 0=\mu \left( x-f(x),%
\frac{t}{a^{k}}\right)
\end{eqnarray*}

Since $\frac{t}{a^{k}}\rightarrow \infty $ for $k\rightarrow \infty ,$ by
means of (vii) and (xiii) properties of intuitionistic fuzzy norm, we see
intuitionistic fuzzy weak contraction map has approximate fixed point
property.
\end{proof}

In the following, we give definition of approximate fixed point property of
a set. Furthermore, we prove that a dense set of intuitionistic fuzzy Banach
space has approximate fixed point property.

\begin{definition}
Let $X$ be IFNS and let $K$ be subset of $X$. Then $K$ is said to have
intuitionistic fuzzy approximate fixed point property (ifafp) if every
intuitionistic fuzzy nonexpansive map $f:K\rightarrow K$ satisfies the
property that $\sup \left\{ \mu \left( x-f\left( x\right) ,t\right) :x\in
K\right\} =1$ and $\inf \left\{ \nu \left( x-f\left( x\right) ,t\right)
:x\in K\right\} =0.$
\end{definition}

\begin{theorem}
Let $X$ be an intuitionistic fuzzy normed space having ifafp, $K$ be dense
subset of $X$. Then $K$ has ifafpp.
\end{theorem}

\begin{proof}
Let $f:X\rightarrow X$ be an intuitionistic fuzzy nonexpansive mapping.
Firstly we prove that%
\begin{equation*}
\sup \left\{ \mu \left( x-f\left( x\right) ,t\right) :x\in K\right\} =\sup
\left\{ \mu \left( y-f\left( y\right) ,s\right) :y\in X\right\}
\end{equation*}%
and%
\begin{equation*}
\inf \left\{ \nu \left( x-f\left( x\right) ,t\right) :x\in K\right\} =\inf
\nu \left\{ \left( y-f\left( y\right) ,s\right) :y\in X\right\}
\end{equation*}%
for $t,s>0.$ Since $K\subset X,$%
\begin{equation*}
\sup \left\{ \mu \left( y-f\left( y\right) ,s\right) :y\in X\right\} \geq
\sup \left\{ \mu \left( x-f\left( x\right) ,t\right) :x\in K\right\}
\end{equation*}%
and%
\begin{equation*}
\inf \left\{ \nu \left( y-f\left( y\right) ,s\right) :y\in X\right\} \leq
\inf \left\{ \nu \left( x-f\left( x\right) ,t\right) :x\in K\right\} .
\end{equation*}%
Let $y\in X.$ There exists a sequence $\left( y_{k}\right) $ in $K$ such
that $y_{k}\overset{\left( \mu ,\nu \right) }{\rightarrow }y$ \ for all $%
y\in X$ because of $K$ is dense. We know that for each $k\in
\mathbb{N}
$ and $t,s>0,$
\begin{eqnarray*}
\sup \left\{ \mu \left( x-f\left( x\right) ,t\right) :x\in K\right\} &\geq
&\mu \left( y_{k}-f\left( y_{k}\right) ,t\right) \\
&\geq &\mu \left( y_{k}-y+y-f\left( y\right) +f\left( y\right) -f\left(
y_{k}\right) ,s\right) \\
&\geq &\mu \left( y_{k}-y,\frac{t}{3}\right) \ast \mu \left( y-f\left(
y\right) ,\frac{t}{3}\right) \ast \mu \left( y_{k}-f\left( y_{k}\right) ,%
\frac{t}{3}\right)
\end{eqnarray*}%
and
\begin{eqnarray*}
\inf \left\{ \nu \left( x-f\left( x\right) ,t\right) :x\in K\right\} &\leq
&\nu \left( y_{k}-f\left( y_{k}\right) ,t\right) \\
&\leq &\nu \left( y_{k}-y+y-f\left( y\right) +f\left( y\right) -f\left(
y_{k}\right) ,t\right) \\
&\leq &\nu \left( y_{k}-y,\frac{t}{3}\right) \Diamond \nu \left( y-f\left(
y\right) ,\frac{t}{3}\right) \Diamond \nu \left( y_{k}-f\left( y_{k}\right) ,%
\frac{t}{3}\right) .
\end{eqnarray*}%
Since $f$ is intuitionistic fuzzy nonexpansive mapping, it is intuitionistic
fuzzy continuous. Because, if $y_{k}\overset{\left( \mu ,\nu \right) }{%
\rightarrow }y,$ then
\begin{eqnarray*}
\mu \left( f\left( y_{k}\right) -f\left( y\right) ,t\right) &\geq &\mu
\left( y_{k}-y,t\right) \rightarrow 1 \\
\nu \left( f\left( y_{k}\right) -f\left( y\right) ,t\right) &\leq &\nu
\left( y_{k}-y,t\right) \rightarrow 0.
\end{eqnarray*}

So $f\left( y_{k}\right) \overset{\left( \mu ,\nu \right) }{\rightarrow }%
f\left( y\right) $ when $y_{k}\overset{\left( \mu ,\nu \right) }{\rightarrow
}y$. If we take limit above inequalities, we get%
\begin{equation*}
\sup \left\{ \mu \left( x-f\left( x\right) ,t\right) :x\in K\right\} \geq
\mu \left( y-f\left( y\right) ,\frac{s}{3}\right)
\end{equation*}%
and%
\begin{equation*}
\inf \left\{ \nu \left( x-f\left( x\right) ,t\right) :x\in K\right\} \leq
\nu \left( y-f\left( y\right) ,\frac{s}{3}\right)
\end{equation*}%
for all $y\in X$ \ and $t,s>0$. Thus, if we take $\frac{s}{3}=s^{\prime },$
\begin{equation*}
\sup \left\{ \mu \left( x-f\left( x\right) ,t\right) :x\in K\right\} \geq
\sup \left\{ \mu \left( y-f\left( y\right) ,s^{\prime }\right) :y\in
X\right\}
\end{equation*}%
and
\begin{equation*}
\inf \left\{ \nu \left( x-f\left( x\right) ,t\right) :x\in K\right\} \geq
\inf \left\{ \nu \left( y-f\left( y\right) ,s^{\prime }\right) :y\in
X\right\} .
\end{equation*}%
Therefore our claim is proved. Now consider any intuitionistic fuzzy
nonexpansive mapping $f_{K}:K\rightarrow K.$ Since $K$ is dense, there
exists a sequence $\left( y_{k}\right) $ in$\ K$ such that $y_{k}\overset{%
\left( \mu ,\nu \right) }{\rightarrow }y$ for any $y\in X.$ Since an
intuitionistic fuzzy nonexpansive mapping is continuous, $f_{K}:K\rightarrow
K$ is intuitionistic fuzzy continuous and it can be extending by defining $%
f\left( x\right) =\lim \left( \mu ,\nu \right) -f\left( x_{k}\right) $ on $%
X. $ Hence we can consider $f$ as an intuitionistic fuzzy nonexpansive
mapping on $X.$ Because, using Lemma 2, we get

\begin{equation*}
\mu \left( f\left( x\right) -f\left( y\right) ,t\right) =\underset{%
k\rightarrow \infty }{\lim }\sup \mu \left( f\left( x_{k}\right) -f\left(
y_{k}\right) \right) ,t\geq \underset{k\rightarrow \infty }{\lim }\sup \mu
\left( x_{k}-y_{k},t\right) =\mu \left( x,y,t\right) \text{ }
\end{equation*}%
and $.$%
\begin{equation*}
\nu \left( f\left( x\right) -f\left( y\right) ,t\right) =\underset{%
k\rightarrow \infty }{\lim \inf }\nu \left( f\left( x_{k}\right) -f\left(
y_{k}\right) \right) \leq \underset{k\rightarrow \infty }{\lim }\inf \nu
\left( x_{k}-y_{k},t\right) =\nu \left( x,y,t\right)
\end{equation*}%
for all $x,y\in X$ and $t>0.$ Then, since $X$ has ifafpp
\begin{equation*}
\sup \left\{ \mu \left( x-f\left( x\right) ,t\right) :x\in K\right\} =\sup
\left\{ \mu \left( y-f\left( y\right) ,t\right) :y\in X\right\} =1
\end{equation*}%
and%
\begin{equation*}
\inf \left\{ \nu \left( x-f\left( x\right) ,t\right) :x\in K\right\} =\inf
\nu \left\{ \left( y-f\left( y\right) ,t\right) :y\in X\right\} =0.
\end{equation*}%
That is, for given any intuitionistic fuzzy nonexpansive mapping $f$ on $K$
we have $\sup \left\{ \mu \left( x-f\left( x\right) ,t\right) :x\in
K\right\} =1$ and $\inf \left\{ \nu \left( x-f\left( x\right) ,t\right)
:x\in K\right\} =0$ and $K$ has (ifafpp)$.$
\end{proof}

\begin{acknowledgement}
This work is supported by Yildiz Technical University Scientific Research
Projects Coordination Unit under the project number BAPK 2012-07-03-DOP03.
\end{acknowledgement}

\end{document}